\documentclass[compmod]{rsl}
\gridframe{N}

\usepackage{amsmath, amssymb, multicol}
\usepackage{cite}
\usepackage{url}
\usepackage{microtype}
\usepackage{newtxtext}      
\usepackage{newtxmath}
\usepackage{color}

\volume{0}
\issue{0}

\pyear{202X}
\pmonth{Month}
\doinu{}
\title[Review of Symbolic Logic]{How Much Propositional Logic Suffices for Rosser's Essential Undecidability Theorem?}
\author[G. Badia, P. Cintula, P. H\'ajek, and A.Tedder]
{GUILLERMO BADIA \\ School of Historical and Philosophical Inquiry, University of Queensland \\ PETR CINTULA, PETR H\'AJEK, and ANDREW TEDDER \\ The Institute of Computer Science of the Czech Academy of Sciences \\\footnote{Received August 2019}}

\leftrunninghead{Badia, Cintula, H\'ajek, and Tedder}
\rightrunninghead{How Much Propositional Logic Suffices for Essential Undecidability?}

\newcommand{\logic}[1]{\ensuremath{\mathrm #1}}
\newcommand{\lang}[1]{\ensuremath{\mathcal #1}}
\newcommand{\LOG}{\ensuremath{\mathfrak{L}}}
\newcommand{\Log}{\ensuremath{\logic{L}_0}}
\newcommand{\PLog}{\ensuremath{\logic{{QL}}_0}}
\newcommand{\DSL}{\ensuremath{\logic{{SL}}_w}}
\newcommand{\PDSL}{\ensuremath{\logic{{{Q{SL}}}_w}}}
\newcommand{\plogic}[1]{\ensuremath{\mathrm{Q}\logic{#1}}}

\newcommand{\on}{\ensuremath{\overline{n}}}
\newcommand{\om}{\ensuremath{\overline{m}}}
\newcommand{\ok}{\ensuremath{\overline{k}}}
\newcommand{\0}{\ensuremath{\overline{0}}}
\newcommand{\1}{\ensuremath{\overline{1}}}
\newcommand{\f}{\ensuremath{\varphi}}
\newcommand{\p}{\ensuremath{\psi}}
\renewcommand{\x}{\ensuremath{\chi}}
\newcommand{\fm}{\emph{Fm}}
\newcommand{\Pow}{\ensuremath{\wp}}
\newcommand{\wdash}{\mathbin{\rhd}}
\newcommand{\onn}{\ensuremath{\overline{n+1}}}

\newcommand{\all}{\ensuremath{\forall}}
\newcommand{\exi}{\ensuremath{\exists}}
\newcommand{\All}[1]{\ensuremath{(\all{#1})}}
\newcommand{\Exi}[1]{\ensuremath{(\exi{#1})}}
\newcommand{\vectn}[1]{\ensuremath{#1_1,\ldots,#1_n}}

\newcounter{qcounter} 

\newenvironment{prooflist}
{\begin{list}{\hbox to 10pt{\texttt{\alph{qcounter})}}}{\usecounter{qcounter}
\setlength\itemsep{1pt}
\setlength\parskip{0pt}
\setlength\parsep{0pt}
\setlength\topsep{0.4ex}
\setlength\labelwidth{25pt}
\addtolength\labelsep{5pt}
\setlength\leftmargin{22pt}
}}
{\end{list}\smallskip}

{\rm \begin{trivlist}\item\renewcommand{\arraystretch}{1.2}\begin{tabular}{l@{\quad}l@{\qquad}l}}
    {\end{tabular}\end{trivlist}}

\begin{document}

\maketitle

\begin{abstract}
In this paper we explore the following question: how weak can a logic be for Rosser's essential undecidability result to be provable for a weak arithmetical theory? It is well known that Robinson's $Q$ is essentially undecidable in intuitionistic logic, and P. H\' ajek proved it in the fuzzy logic BL for Grzegorczyk's variant of $Q$ which interprets the arithmetic operations as non-total non-functional relations. We present a proof of essential undecidability in a much weaker substructural logic and for a much weaker arithmetic theory, a version of Robinson's $R$ (with arithmetic operations also interpreted as mere relations). Our result is based on a structural version of the undecidability argument introduced by Kleene and we show that it goes well beyond the scope of the Boolean, intuitionistic, or fuzzy logic.
\end{abstract}

\footnotetext{This paper is an extension and generalisation of work started by Cintula and H\'{a}jek before the unfortunate death of the latter in 2016. H\'{a}jek is included among the authors in recognition of his work on this topic, and with the blessing of his family, but it should be noted that he was not able to contribute directly to the final version of this paper.}

\section{Introduction}

In Theorem III of Rosser (1936), it was famously established that Peano Arithmetic was \emph{essentially undecidable} (a notion only properly named later by Tarski (1949)); that is, no consistent extension of it is decidable (see Tarski et al (1953) for the standard reference on this topic). After Rosser's essential undecidability theorem, it was natural to ask for weaker theories of arithmetic that would still yield undecidability along similar lines. Robinson (1950) provided the perhaps best known example of such a theory, namely, Robinson's Arithmetic $Q$.

The most noteworthy other essentially undecidable weakening of $Q$ which will play a special role here is $R$, also due to Robinson (see Tarski et al (1953), p.~53\footnote{See Visser (2014) for a survey of results involving $R$ and Vaught (1962) for the original reference regarding undecidability of this theory. Further noteworthy results on $R$ are given by Jones and Shepherdson (1983).}), which allows for the so-called \emph{structural} essential undecidability argument (to borrow the terminology from \v{S}vejdar (2008)) due originally to Kleene (1950); see Proposition~15.9 and Theorem~15.19 of Monk (1976) for the textbook version of the argument, or below for our rendering. Below are the standard axioms of $R$ and $Q$:\footnote{In the axiomatisation of $R$, $\leq$ is often taken to be a defined predicate, which allows for ($R$4) to be simplified to just the left-to-right direction of our biconditional version. Since we shall take $\leq$ as primitive, we include both directions. Note also that we use $n$ to refer to a number, and $\on$ to refer to the associated numeral (this will be properly defined in Section~\ref{s:MainResult}).}

\pagebreak

%

\begin{center}
\begin{tabular}{@{\hspace{-17ex}}l@{\hspace{-40ex}}c}
\begin{tabular}{l@{\ \ }l}
$(R1)$ & $\om +\on= \overline{m+n}$
\\ $(R2)$ & $\om\cdot\on= \overline{m \cdot n}$
\\ $(R3)$ & $\om\neq \on$ \,\,\,\, for $m\neq n$
\\ $(R4)$ & $x\leq \on \leftrightarrow (x=\0 \vee x=\1 \vee \dots \vee x=\on)$
\\ $(R5)$ & $x\leq \overline{n } \lor \on\leq x$
\end{tabular}
&
\begin{tabular}{l@{\ \ }l}
$(Q1)$ & $S(x)\neq\0$
\\ $(Q2)$ & $S(x)=S(y)\to x=y$
\\ $(Q3)$ & $x\neq\0\to(\exists y)(x=S(y))$
\\ $(Q4)$ & $x+\0=x$
\\ $(Q5)$ & $x+S(y)=S(x+y)$
\\ $(Q6)$ & $x\cdot \0=\0$
\\ $(Q7)$ & $x\cdot S(y)=(x\cdot y)+x$
\\ $(Q8)$ & $x\leq y\leftrightarrow(\exists z)x+z=y$
\end{tabular}
\end{tabular}
\end{center}

Given the proliferation of non-classical logical systems in the literature, the question not only of potential weakenings of the arithmetic theory but also of the background \emph{propositional} logic become salient, and our aim here is to consider how much (or how little) propositional logic suffices for something like Kleene's argument. Our strategy will consist in a close inspection of the structural essential undecidability argument in order to generalise it to the non-classical case, with an eye to the question: what logical principles are actually required for the argument to work and in which non-classical settings are they available?

It was already well-known at least since the 1950s (see Kleene (1952)) that the undecidability results for $Q$ hold intuitionistically as well as classically. H\'ajek (2007) showed that it was also true for a wide variety of fuzzy logics and an arithmetic theory called $Q^\backsim$, a variant of $Q^-$, introduced by Grzegorczyk (2006) (see also \v{S}vejdar (2007)) which results from $Q$ by replacing the addition and multiplication functions by ternary predicates $A$ and $M$.

\interfootnotelinepenalty=1

We shall show that even against the background of a weaker logic than H\'ajek ever considered, we can prove essential undecidability for a cousin (in fact, a weakening against a certain logical background) of $Q^\backsim$ that we call $R^\backsim$, which is a natural generalization of $R$. $R^\backsim$ and $Q^\backsim$  are axiomatised as follows:

\begin{itemize}
\itemsep=0em
\item[] $(R^\backsim1)$ \; $A(\om,\on,x)\leftrightarrow \overline{m+n}=x$
\item[] $(R^\backsim2)$ \; $M(\om,\on,x)\leftrightarrow \overline{m \cdot n}=x$
\item[] $(R^\backsim3)$ \; $\om\neq \on$ \,\,\,\, for $m\neq n$
\item[] $(R^\backsim4)$ \; $x\leq \on \leftrightarrow (x=\0 \vee x=\1 \vee \dots \vee x=\on)$
\item[] $(R^\backsim5)$ \; $x\leq \overline{n } \lor \on\leq x$
\item[] $(R^\backsim6)$ \; $x\leq\on\lor\neg(x\leq\on)$
\end{itemize}
\begin{itemize}
\itemsep=0em
\item[] $(Q^\backsim0)$ \; $x=y\lor x\neq y$
\item[] $(Q^\backsim1)$ \; $S(x)\neq\0$
\item[] $(Q^\backsim2)$ \; $S(x)=S(y)\to x=y$
\item[] $(Q^\backsim3)$ \; $x\neq\0\to(\exists y)(x=S(y))$
\item[] $(Q^\backsim4)$ \; $A(x, \0, y)  \leftrightarrow  x=y$
\item[] $(Q^\backsim5)$ \; $A(x, S(y), z)  \leftrightarrow  (\exists u)(A(x, y, u) \wedge z=S(u) )$
\item[] $(Q^\backsim6)$ \; $M(x, \0, y)  \leftrightarrow  y=\0$
\item[] $(Q^\backsim7a)$ $ M(x, S(y), z)\to (\exists u)(M(x, y, u) \wedge A(u, x, z))$
\item[] $(Q^\backsim7b)$ $M(\om,\on,u)\to(A(u,\on, x)\to M(\om,\overline{n+1},x))$
\item[] $(Q^\backsim8)$ \; $x\leq y \leftrightarrow  (\exists z) A(z, x, y)$
\end{itemize}

\begin{remark}
A brief comment is in order concerning concerning these axioms. First note that our additional axiom ($R^\backsim$6) is an instance of excluded middle. We include ($R^\backsim$6) because the weaker of our two logics does not allow its derivation from the others (though see Prop.~\ref{fund}, in which we derive it from the other $R^\backsim$ axioms in our stronger logic (along with the assumption of excluded middle for equalities)). Furthermore, it should be noted that adding the assumption of functionality and totality of $A,M$ to $R^\backsim$ or $Q^\backsim$ results in $R$, $Q$ respectively, in classical logic. Finally, in Section~\ref{s:QandR} we comment further on the relation of our axiomatisation $Q^\backsim$ with the system studied by H\'{a}jek.
\end{remark}

In Section~\ref{s:StructProof} we present the necessary basic definitions in an abstract setting and outline the aforementioned \emph{structural} essential undecidability argument. We shall see that the crucial ingredient of the proof is the existence of formulas separating two disjoint recursively enumerable sets: call such a formula one which \emph{strongly separates} the sets in question. In Section~\ref{s:logic} we present our weak logic and prove the existence of strongly separating formulae for $R^\backsim$ against the background of this logic in Section~\ref{s:MainResult} (after establishing its completeness with respect to the standard model of arithmetic for $\Sigma_1$-formulas) and finally in Section~\ref{s:QandR} we show that, in a slightly stronger logical setting (presented in Section~\ref{s:stronger-logic}), a variant of H\' ajek's $Q^\backsim$ strengthens our $R^\backsim$, and thus our results indeed generalize those of H\'ajek (2007).

\section{The Structural Proof of Essential Undecidability}\label{s:StructProof}

The first ingredients we need are the \emph{formulas.} As we do not yet want to bind ourselves to any particular syntax (propositional or first order), let us only assume that we have a countable set of formulas $\fm$ which contains a special subset of formulas that we will suggestively call the $\Sigma_1$-formulas and that for each $\Sigma_1$-formula $\f$ there is a special formula (not necessarily a $\Sigma_1$-formula) which we call the negation of $\f$ and suggestively denote $\neg\f$.\footnote{Note that $\neg\f$ need not be a formula \emph{per se,} it is merely a notation of the negation of $\f$. The argument below would work perfectly well if all formulas would be $\Sigma_1$-formulas; we however in principle assume it can be a proper subset to cover a wider logical setting. See mainly our notion of consistency below.}

The second ingredient is that of a {\em logic} $\LOG$ which is identified with a consequence relation $\vdash_\LOG$ over $\fm$, i.e.,  ${\vdash_\LOG} \subseteq \Pow(\fm)\,\times\, \fm$ and for each $\Gamma\cup\Delta\cup\{\f\}\subseteq \fm$ we have:
\begin{itemize}
\item $\Gamma\cup\{\f\}\vdash_\LOG \f$ \hfill (Reflexivity)
\item If $\Gamma\vdash_\LOG \f$ and for each $\gamma\in\Gamma$ we have $\Delta\vdash_\LOG \gamma$, then $\Delta  \vdash_\LOG \f$ \hfill (Cut)
\end{itemize}

By \emph{theory} we understand simply \emph{a set of formulas}; for each theory $T$ we define the set of its consequences in logic $\LOG$ as $C_\LOG(T) = \{\f\mid T\vdash_\LOG \f \}$.

As the final ingredient we need to define the notion of \emph{essential undecidability} of a theory. As the underlying logic can vary, we  have to be a bit more careful and formalistic in our definitions of (un)decidability, extension and consistency now (so that we can be a bit looser going forward).\footnote{The abstract framework we develop here can be seen as a still less general version of a similar framework developed by Smullyan (1961) for abstract reasoning about undecidability and incompleteness. We could use his framework of \emph{representation systems} in order to present our results by fixing for one of our arithmetic theories $U$, a representation where $S$ is the set of $\Sigma_1$-formulae in the language of $U$, $T=C_\LOG(U)$, and $R=\{\f\mid U\vdash_{\LOG}\neg\f\}$. This would put our work into Smullyan's context, but we follow our current mode of presentation here because we do not need the additional generality provided by Smullyan's approach here. Thanks are due to an anonymous referee for pointing out this connection.} When we speak about the decidability of a theory $T$, we actually speak about the decidability of the set $C_\LOG(T)$, i.e., questions of decidability depends on the logic in question (e.g.\ all theories are trivially decidable in the inconsistent logic $\text{Inc} = \Pow(\fm)\times \fm$). Analogously when we say that a theory $T$ strengthens a theory $S$ we do not speak about simple subsethood but about the fact that $T$ proves all axioms of $S$ in $\LOG$, i.e., $S\subseteq C_\LOG(T)$. Finally the consistency depends on the logic in question and on the class of $\Sigma_1$-formulas: we say that a theory $T$ is \emph{$\Sigma_1$-consistent in $\LOG$} if for no $\Sigma_1$-formula $\f$ we have $T\vdash_\LOG \f$ and $T\vdash_\LOG \neg\f$ (note the our notion of consistency implies non-triviality, i.e., that there is a $\f$ such that $T\nvdash_\LOG \f$, and in sufficiently strong logics the converse is true as well; furthermore in a sufficiently strong arithmetical theory it is equivalent with $T\nvdash_\LOG \0=\1$).

Therefore we should speak about $\Sigma_1$-consistency, decidability, and strengthening \emph{in} $\LOG$. To simplify matters, whenever the logic is known from the context we assume that all subsequent uses of these three notions are parameterized by the logic in question. We also omit the prefix $\Sigma_1$, when this is clear from the context. With this convention in place, we can give the following definition analogously to the classical case:

\begin{definition}
A theory $T$ is \emph{essentially undecidable in~$\LOG$} if it is consistent and each consistent theory strengthening $T$ is undecidable.
\end{definition}

The next proposition shows that our notion of essential undecidability is a particularly robust version of essential undecidability as it
is preserved not only in stronger theories but also in stronger logics. The statement of  this proposition is made somewhat intricate to accommodate for the possibility of a stronger logic being defined over a bigger set of formulas.

\begin{prop}\label{p:UndecUpward}
Let $\LOG$ be a logic over a set of formulas $\fm$ and $\LOG'\supseteq \LOG$ a logic over a set of formulas $\fm'\supseteq\fm$ such that all $\Sigma_1$-formulas of $\fm$ are $\Sigma_1$-formulas of $\fm'$. Furthermore, assume that $T$ is an essentially undecidable theory in~$\LOG$. Then any theory $S$ of $\LOG'$ which is consistent and strengthens $T$ (seen as a $\LOG'$-theory) is essentially undecidable in~$\LOG'$.
\end{prop}

\begin{proof}
It suffices to show that  any theory $U$ which is consistent in $\LOG'$ and strengthens $S$ in $\LOG'$ is undecidable in $\LOG'$. Define a set $V = \{\x \in \fm\mid U \vdash_{\LOG'} \x\}$ and observe $V = C_{\LOG}(V)$: indeed if $V \vdash_\LOG \delta$ implies $V \vdash_{\LOG'} \delta$ and for each $\x \in V$ we have $U\vdash_{\LOG'} \x$ and so due to the (cut) rule of $\LOG'$ we have $U\vdash_{\LOG'} \delta$, i.e., $\delta \in V$.

Therefore $V$ is consistent in $\LOG$ (otherwise we would obtain contradiction with the assumption that $U$ is consistent in $\LOG'$) and strengthens $T$ in $\LOG$ (actually $T \subseteq U$) and so by essential undecidability of $T$ in~$\LOG$ we know that $V$ is undecidable. Since $\x\in V$ iff $\x\in C_{\LOG'}(U)$, then $C_{\LOG'}(U)$ is undecidable as well.
\end{proof}

When we establish that a theory $T$ of a logic $\LOG$ is a Rosser theory, defined below, the proof of the essential undecidability theorem can proceed as in the classical setting. We present the proof in some detail so it is obvious that no additional properties of $\LOG$ are needed.

\begin{definition}[Rosser Theories]
We say that a theory $T$ is \emph{Rosser} in logic $\LOG$ if for each pair of disjoint recursively enumerable sets $A,B \subseteq \mathbb{N}$ there is a recursive series of $\Sigma_1$-formulas $\f_n$ which \emph{strongly separate} $A$ from $B$, that is:
\end{definition}

\begin{itemize}
\item $n\in A$ \emph{implies} $T \vdash_\LOG \f_n$.
\item $n\in B$ \emph{implies} $T \vdash_\LOG \neg\f_n$.
\end{itemize}

\begin{thm}[Rosser Theorem]\label{t:RosGen}
Let $\LOG$ be a logic and $T$ an $\LOG$-consistent Rosser theory. Then $T$ is essentially undecidable in $\LOG$.
\end{thm}

\begin{proof}
First recall that from recursion theory (see e.g. Theorem~6.24 of Monk (1976)) that we know that there are disjoint r.e.\ sets $A,B \subseteq \mathbb{N}$ such that each $X\supseteq A$ such that $X\cap B = \emptyset$ is not recursive. Let $\f_n$ be the series of $\Sigma_1$-formulas guaranteed strongly separating $A$ and $B$.

\pagebreak

Consider a consistent theory $S$ strengthening $T$, define $X = \{n \mid S \vdash \f_n\}$, and notice that:
\begin{itemize}
\item $n\in A$ implies $T \vdash_\LOG \f_n$ and so, due to the (cut) rule, $S \vdash \f_n$, which entails $X\supseteq A$.
\item $n\in B$ implies $T \vdash_\LOG \neg\f_n$ and so $S \vdash \neg\f_n$, which entails $X\cap B = \emptyset$ (otherwise $S$ would not be consistent).
\end{itemize}
Thus $X$ cannot be recursive and as it is clearly recursively reducible to $C_\LOG(S)$, $S$ cannot be decidable in~$\LOG$.
\end{proof}

This structural argument leaves open the question of which logical principles are needed to prove that some theory $T$ is Rosser in $\LOG$, which is where our work really begins.

\section{A weaker logic}\label{s:logic}

The minimal logic in which we prove that $R^\backsim$ is Rosser will be defined as a natural first-order extension of a particular propositional  non-classical logic. We take \emph{propositional} logics to be substitution-invariant consequence relations over a set of propositional formulas given by a \emph{propositional} language $\lang{L}$ (a set of connectives with arities). It is well-known that each such logic can be presented by means of a Hilbert style proof system consisting of axiom and rule schemata (we use the $\wdash$ symbol to separate premises of a rule from its conclusion).


Our basic propositional language $\mathfrak{L}_0$ consists of three binary connectives (implication $\to$, (lattice) conjunction $\wedge$ and (lattice) disjunction $\vee$) and a propositional constant $\bot$.\footnote{Note that in our weaker logic (axiomatised below) we do not include the axiom $\bot\rightarrow\f$, so though we use this suggestive notation, for now $\bot$ is merely a propositional constant (however later in Section~\ref{s:stronger-logic}, we will will work with the logic \DSL\ obeying this additional axiom, in which it will become a genuine \emph{falsum} constant).} Negation $\neg$ and equivalence $\leftrightarrow$ are defined in the following standard way:
$$
\neg\f:=\f\to\bot \qquad\qquad \f\leftrightarrow\p:=(\f\to\p)\wedge(\p\to\f).
$$

As usual, we assume that $\neg$ has the highest binding power, followed by $\wedge$ and $\vee$, and finally $\to$ and $\leftrightarrow$ have the lowest. Our basic propositional logic in the language $\lang{L}_0$, which we denote as $\Log$, is given by the following Hilbert style proof system:\footnote{Note the we use the same symbol for two closely related yet different axioms, we can afford this slight abuse of language as in any given formal proof it it will be clear which of them we are using. We use the same convention for some other upcoming axiom/theorems/rules.}

\medskip

\begin{tabular}{l@{\quad}l@{\qquad\quad}l@{\quad}l}
 (identity) & $\f\to\f$ & (weakening) & $\f\wdash \p\to\f$ 
\\ ($\wedge$elim) & $\f\wedge\p\to\f$  &  (MP) & $\f,\f\to\p\wdash\p$
\\ ($\wedge$elim) & $\f\wedge\p\to\p$  &  (assertion) & $\f\wdash(\f\to\p)\to\p$
\\ ($\vee$intro)& $\f\to \f\vee\p$  & (trans) & $\f\to\p,\p\to \x\wdash\f\to\x$
\\ ($\vee$intro)& $\p\to \f\vee\p$ & (morg) &  $\neg \f \vee \neg \p \wdash \neg (\f \wedge \p)$
\\ && ($\wedge$intro) & $\x\to\f, \x\to\p \wdash \x\to\f\wedge\p$
\\ && ($\vee$elim)& $\f\to\x, \p\to\x \wdash \f\vee\p\to\x$
\end{tabular}

\

\noindent The following theorems and derived rules are easily shown to be derivable in $\Log$:\footnote{Associativity is a simple consequence of ($\vee$intro), ($\vee$elim) and (trans); adjunction follows using (weakening) for $\p$ being any theorem and ($\wedge$intro) together with (MP); and finally (dni) is an instance of (assertion).}

\medskip

\begin{tabular}{l@{\quad}l}
(assoc) & $\f\vee(\p\vee\x) \leftrightarrow (\f\vee\p)\vee\x$
\\ (adj) & $\f,\p\wdash\f\wedge\p$
\\ (dni) & $\f \wdash \neg \neg \f$
\end{tabular}

\begin{remark}\label{comparinglogics}
$\Log$ extends lattice logic by the rules (morg), (weakening), and (assertion). These additional rules are chosen to fulfill specific tasks in fleshing out the structural argument for our version of $R^\backsim$. So if $\Log$ looks somewhat artificial, that's because it is. However, it should be noted that it is a sublogic of many well known systems, most notably classical and intuitionistic logic, as well as H\'ajek's {\rm BL} and (related to the stronger system we'll present later in Section~\ref{s:stronger-logic}) the non-distributive, non-associative Lambek calculus with weakening in $\lang{L}_0$, when this is presented with (assertion) as a rule (related systems to this are discussed, for instance, in Galatos et al (2007)). In fact, it is easy to see that $\Log$ is a proper sublogic of all these logics.
\end{remark}

Now we are ready to introduce first-order logics. We present only the syntactical aspects, excepting where we consider a very special model (namely, the natural numbers with a classical interpretation of the vocabulary). Let us fix a propositional logic $\logic{L}$ expanding $\Log$ (thus in particular, $\logic{L}$ has a propositional language which contains $\lang{L}_0$). Our notion of first-order language is standard, i.e., it is given by a set of function and predicate symbols with a special binary predicate symbol $=$ for equality (we write $t\neq s$ for $\neg(t = s)$).  Then \emph{terms;} \emph{atomic formulas,} and \emph{formulas} are built up as usual. In addition, the notions of free/bounded variable, substitutability, sentence, etc.\/ are defined as usual.

The first-order logic $\plogic{L}$, for a given predicate language, is axiomatised by the substitutional instances of all axioms/rules of $\logic{L}$ (i.e., formulas resulting by replacing atoms by first-order formulas) and the following additional axiom/rule schemata (we assume that $t$ substitutable for $x$ in $\f$ and $x$ not free in $\x$):

\smallskip


\begin{flushleft}
\begin{tabular}{l@{\ \ }l}
($\forall$ins) & \;$\All{x}\f(x) \to \f(t)$                 
\\ ($\exists$intro) & \;$ \f(t)\to\Exi{x}\f(x) $        
\\ ($\forall$intro) & \;$\x\to\p\wdash\x\to(\forall x)\p$      
\\ ($\exists$elim) & \;$\p \to \x\wdash \Exi{x}\,\p \to \x$   
\\    (id) &  \;\;$x=x$
\\ (com) &  \;\;$x=y \to y = x$
\\   (trans) & \; $x=y \to (y = z \to x = z)$
\\  ($=$prin) &\;  $x=y\to t(x) = t(y)$
\\  ($=$prin) & \; $x=y\to(\f(x)\to\f(y))$
\end{tabular}
\end{flushleft}

\medskip

\noindent We can easily establish the following auxiliary derived rules:\footnote{To derive (aux), assuming $\f(t)$, use (assertion) to obtain $(\f(t)\to\f(x))\to\f(x)$ and so ($=$prin) and (trans) complete the proof.}

\smallskip

\begin{flushleft}
\begin{tabular}{l@{\quad}l}
(aux) & $\f(t)\wdash x=t \to\f(x)$ \\
(gen) & $\f\wdash(\forall x)\f$
\end{tabular}
\end{flushleft}

\section{Main result}\label{s:MainResult}

In order to apply the structural argument presented in Section~\ref{s:StructProof} we need to fix a logic over a set of formulas, identify the $\Sigma_1$-formulas, define a theory, show that it is consistent and Rosser.

We will work with the logic $\PLog$ over the set of formulas for the predicate language of Grzegorczyk's arithmetic, i.e., the language with constant $\0$, unary function symbol $S$, binary predicate symbols $=$ and $\leq$, and ternary predicate symbols $A$ and $M$. We define the $n$th numeral $\on$ as usual: $\on= S \stackrel{n}{\dots}S(\0).$ 

\pagebreak

$\Sigma_1$-formulas be those of the form $\f = \Exi{x} \p$ for some $\Delta_0$-formula $\p$, where $\Delta_0$-formulas are those arithmetical formulas where all quantifiers are bounded, i.e., are of the form:\footnote{Let us stress that this definition of bounded quantification is intended for the classical arithmetical formulas; the bounded quantifiers in non-classical logics may be (and often are) defined using implication and strong conjunction; but as we have no need for such quantifiers in the paper, no confusion should arise.}
\begin{align*}
  \All{x \leq y}\f   & = \All{x}(\neg(x\leq y)\vee \f)\\
  \Exi{x \leq y}\f   & = \Exi{x}(x\leq y\wedge \f)
\end{align*}
As expected the negation of a $\Sigma_1$-formula $\f$ will be the formula $\neg\f$. Now we can formally present the theory $R^\backsim$ which we mentioned in the introduction:\footnote{Note that because of (assoc), we may state $(R^\backsim4)$ with no fixed association.}

\smallskip

\begin{flushleft}
\begin{tabular}{l@{\ \ }l@{\quad}l}
$(R^\backsim1)$ &\; $A(\om,\on,x)\leftrightarrow \overline{m+n}=x$ &\; for any $n$ and $m$
\\ $(R^\backsim2)$ &\; $M(\om,\on,x)\leftrightarrow \overline{m \cdot n}=x$ &\; for any $n$ and $m$
\\ $(R^\backsim3)$ &\; $\om\neq \on$ & \;\:for $m\neq n$
\\ $(R^\backsim4)$ &\; $x\leq \on \leftrightarrow (x=\0 \vee x=\1 \vee \dots \vee x=\on)$ &\; for any $n$
\\ $(R^\backsim5)$ &\; $x\leq \on \vee \on\leq x$ &\; for any $n$
\\ $(R^\backsim6)$ &\; $x\leq \on \vee \neg(x\leq \on)$ &\; for any $n$
\end{tabular}
\end{flushleft}

\smallskip

Our goal in this section is to establish essential undecidability of $R^\backsim$ in $\PLog$. Thanks to Theorem~\ref{t:RosGen} we know that it suffices to prove that it is $\PLog$-consistent and Rosser.

By $N$ we denote the set of natural numbers and by $\mathbb{N}$ we denote the standard model of the natural numbers in our arithmetical language, i.e., the structure with constant $\0$ interpreted as $0$, $S$ interpreted so that $S(k) = k+1$, $\leq$ interpreted by the usual order of natural numbers, and $A$ and $M$ as:
\begin{align*}
A^{\mathbb{N}}(k,l, m)  & \text{ iff } k+l = m \\
M^{\mathbb{N}}(k,l, m)  & \text{ iff } k \cdot l = m.
\end{align*}

In this model we can interpret all connectives and quantifiers of $\PLog$ in a fully classical way (so, for instance, $\f\rightarrow\p$ should be interpreted as $\neg\f\lor\p$ for Boolean $\neg$). For a formula $\f(\vectn{x})$ with free variables $\vectn{x}$, we write:

\begin{itemize}
\item $\mathbb{N}\models\f(\vectn{k})$ if $\f$ is satisfied in $\mathbb{N}$ when variables $x_i$ are evaluated as $k_i$
\item $\mathbb{N}\models\f$ if $\mathbb{N}\models\f(\vectn{k})$ for each $\vectn{k}\in N$.
\end{itemize}

Note that a numeral $\on$ is interpreted in $\mathbb{N}$ as the number $n$ and so we have $\mathbb{N}\models\f(\vectn{k})$ iff $\mathbb{N}\models\f(\overline{k_1}, \dots, \overline{k_n})$.

It is easy to see that the structure $\mathbb{N}$ can be interpreted as a model of $R^{\backsim}$ in $\PLog$; formally speaking we can prove the following (as in the rest of this section we work in the logic $\PLog$ only, we omit it as a subscript of $\vdash$):

\begin{prop}[Soundness]\label{p:consR}
For each formula $\f$, $R^{\backsim} \vdash \f$ implies $\mathbb{N}\models \f$.
\end{prop}

For no $\Sigma_1$-formula $\f$ do we have $\mathbb{N}\models \f$ and $\mathbb{N}\models \neg\f$, and this entails that $R^{\backsim}$ is $\PLog$-consistent. As the next step we establish the converse claim for $\Sigma_1$-sentences (known as $\Sigma_1$-completeness), i.e., for each $\Sigma_1$-sentence $\f$ we have $\mathbb{N} \models \f$ only if\/ $R^{\backsim} \vdash \f$. We prove a stronger statement.

\begin{thm}[ $\Sigma_1$-completeness]\label{t:Sigma-comp}
For each $\Sigma_1$-formula $\f(x_1, \dots, x_n)$, we have:
$$
\mathbb{N} \models \f(\vectn{k})\qquad\text{iff}\qquad R^{\backsim} \vdash \f(\overline{k_1}, \dots, \overline{k_n)}.
$$
\end{thm}

\begin{proof}
One direction is  a consequence of Proposition~\ref{p:consR} and the fact that  $\mathbb{N}\models\f(\vectn{k})$ iff $\mathbb{N}\models\f(\overline{k_1}, \dots, \overline{k_n})$. We prove the converse direction first for $\Delta_0$-formulas by induction on the complexity of $\f$. As we cannot deal with negation directly, the induction step for it will have to take care of the next principal connective/quantifier.

\begin{description}
\item[$\bullet$\quad $\f$ is atomic] First note that all terms are of the form ${S \dots S}(x)$ for some variable $x$, so it suffices to prove the claim for numerals. Observe that (1) if $\mathbb{N} \models n = m$, then  $R^{\backsim} \vdash \on = \om$ by ($=$prin); (2) if $\mathbb{N} \models A(n,m,k)$ (i.e., $n + m = k$), then as before $R^{\backsim} \vdash \overline{n+m} = \ok$ which due to $(R^\backsim1)$ implies
    $R^{\backsim} \vdash  A(\on,\om,\ok)$; (3) analogously $\mathbb{N} \models M(n,m,k)$ implies $R^{\backsim} \vdash M(\on,\om,\ok)$ using $(R^\backsim2)$; and finally (4) from $\mathbb{N} \models m \leq n$, we know that $R^{\backsim} \vdash \om \leq \on$ thanks to axiom $(R^\backsim4)$ for $x = \om$ (as then $\om = \om$ is one of the disjoints on the right-hand side disjunction).

\item[$\bullet$\quad $\f = \p\wedge \x$] From the assumption we obtain $\mathbb{N} \models \p$ and $\mathbb{N} \models \x$. Thus by the induction assumption  $R^{\backsim} \vdash \p$ and  $R^{\backsim} \vdash \x$, and so (adj) completes the proof.

\item[$\bullet$\quad $\f = \p\vee \x$] From  the assumption we obtain $\mathbb{N} \models \p$ or $\mathbb{N} \models \x$. Thus by the induction assumption  $R^{\backsim} \vdash \p$ or $R^{\backsim} \vdash \x$, and so ($\vee$intro) completes the proof.

\item[$\bullet$\quad $\f =  \Exi{y \leq x}\p(y, \vectn{x})$] From  $\mathbb{N} \models \Exi{y \leq k}\p(y, \vectn{k})$ we obtain an $m$ such that
$$
\mathbb{N} \models  m \leq k \qquad \text{and}\qquad \mathbb{N} \models \p(m,\vectn{k}).
$$
Thus by the induction assumption $R^{\backsim} \vdash \om \leq \overline{k}$ and $R^{\backsim} \vdash\p(\om,\overline{k_1}, \dots, \overline{k_n})$  and so (adj) and ($\exists$intro) complete the proof.

\item[$\bullet$\quad $\f = \All{y \leq x}\p(y, \vectn{x})$] From  $\mathbb{N} \models \All{y \leq k}\p(y, \vectn{k})$ we know that for each $m\leq k$ we have $\mathbb{N} \models \p(m, \vectn{k})$. Thus by the induction assumption we obtain,
$$
R^{\backsim} \vdash \p(\om, \overline{k_1}, \dots, \overline{k_n}).
$$
Next we use (aux) to obtain, for $m\leq k$:
$$
R^{\backsim} \vdash y = \om  \to \p(y, \overline{k_1}, \dots, \overline{k_n}).
$$
Thus, using ($\vee$elim), we obtain
$$
R^{\backsim} \vdash (y= \0 \vee y = \1 \vee \dots \vee y =\ok) \to \p(y, \overline{k_1}, \dots, \overline{k_n})
$$
which, using $(R^\backsim4)$, ($\vee$intro), and (trans), entails:
$$
R^{\backsim} \vdash (y \leq \overline{k}) \to \neg(y \leq \overline{k})\vee\p(y, \overline{k_1}, \dots, \overline{k_n}).
$$
As clearly using ($\vee$intro) we also have
$$
R^{\backsim} \vdash \neg(y \leq \overline{k}) \to \neg(y \leq \overline{k})\vee\p(y, \overline{k_1}, \dots, \overline{k_n}).
$$
then by ($\vee$elim) and axiom $(R^\backsim6)$,
$$
R^{\backsim} \vdash \neg(y \leq \overline{k})\vee\p(y, \overline{k_1}, \dots, \overline{k_n}).
$$
and (gen) complete the proof.

\pagebreak

\item[$\bullet$\quad $\f = \neg \p$] We have to distinguish the structure of $\p$:

\begin{description}
\item[$-$\quad $\p$ is atomic] Analogously to the positive case: (1) If $\mathbb{N} \models n \neq m$, then  $R^{\backsim} \vdash \neg (\on = \om)$ by $(R^\backsim3)$; (2) if $\mathbb{N} \models \neg A(n,m,k)$ (i.e., $n + m \neq k$), then as before $R^{\backsim} \vdash  \overline{n+m} \neq \ok$ which due to $(R^\backsim1)$ and (trans) implies $R^{\backsim} \vdash  \neg A(\on,\om,\ok)$; (3) analogously $\mathbb{N} \models \neg M(n,m,k)$ implies $R^{\backsim} \vdash \neg M(\on,\om,\ok)$ using $(R^\backsim2)$; and finally (4) from $\mathbb{N} \models \neg (m \leq n)$, implies for each $k \leq n$ we have $k\neq m$, thus by $(R^\backsim3)$ we obtain
$$
 R^{\backsim} \vdash  \om = \ok \to \bot.
$$
Thus by ($\vee$elim) we obtain
 $$
 R^{\backsim} \vdash  \om = \0 \vee \om = \1 \dots \vee \om = \on \to \bot
$$
and so $(R^\backsim4)$ and and (trans) complete the proof.

\item[$-$\quad $\p = \alpha \wedge \beta$] From $\mathbb{N} \models \neg(\alpha \wedge \beta)$ we obtain $\mathbb{N} \models \neg \alpha$ or  $\mathbb{N} \models \neg \beta$. Thus by the induction assumption we know that $R^{\backsim} \vdash \alpha \to \bot$ or $R^{\backsim} \vdash \beta \to \bot$ and so in both cases ($\wedge$elim) and (trans) completes the proof.

\item[$-$\quad $\p = \alpha \vee\beta$] From $\mathbb{N} \models \neg(\alpha \vee \beta)$ we obtain $\mathbb{N} \models \neg \alpha$ and  $\mathbb{N} \models \neg \beta$. Thus by the induction assumption we know that $R^{\backsim} \vdash  \alpha\to \bot $ and $R^{\backsim} \vdash \beta\to\bot$ and so ($\vee$elim) completes the proof.

\item[$-$\quad $\p = \neg \x$] From $\mathbb{N} \models \neg\neg\x$ we obtain $\mathbb{N} \models \x$. Thus by the induction assumption we know that  $R^{\backsim} \vdash  \x$ and so (dni) completes the proof.

\item[$-$\quad $\p = \Exi{y \leq x}\x(y, \vectn{x})$] From $\mathbb{N} \models \neg \Exi{y \leq k}\x(y, \vectn{k})$ we obtain for each $m \leq k$ that: $\mathbb{N} \models \neg\x(m, \vectn{k})$ . Thus as in the the positive case for $\forall$  we could show that
$$
R^{\backsim} \vdash \neg(y \leq \overline{k})\vee\neg\x(y, \overline{k_1}, \dots, \overline{k_n}).
$$
Thus by (morg) we obtain
$$
R^{\backsim} \vdash y \leq \overline{k}\wedge\x(y, \overline{k_1}, \dots, \overline{k_n}) \to \bot
$$
and so ($\exists$elim) completes the proof.

\item[$-$\quad $\p =  \All{y \leq x}\x(y, \vectn{x})$]  From $\mathbb{N} \models \neg \All{y \leq k}\x(y, \vectn{k})$ we know that there is an $m$ such that
$$
\mathbb{N} \models  \neg\neg( m \leq k) \qquad \text{and}\qquad \mathbb{N} \models \neg\x(m,\vectn{k}).
$$
Thus by the induction assumption  we obtain
$$
R^{\backsim} \vdash \neg(\om \leq \ok)\to \bot \qquad \text{and}\qquad  R^{\backsim} \vdash \x(\om,\overline{k_1}, \dots, \overline{k_n}) \to \bot
$$
which using ($\vee$elim) entails
$$
R^{\backsim} \vdash (\neg(\om \leq \ok)\vee\x(\om, \overline{k_1}, \dots, \overline{k_n})) \to \bot
$$
and so ($\forall$ins) and (trans) complete the proof.
\end{description}
\end{description}

Finally we deal with $\Sigma_1$-formulas. Assume that $\mathbb{N} \models \Exi{y}\p$ for some $\Delta_0$-formula $\p$. Then there is some $m$ such that $\mathbb{N} \models \p(\om, \overline{k_1}, \dots, \overline{k_n})$. Hence, $R^{\backsim} \vdash \p(\om, \overline{k_1}, \dots, \overline{k_n})$ and so ($\exists$intro) completes the proof.
\end{proof}

Now we have all the ingredients to prove that $R^\backsim$ is Rosser in $\PLog$; we will actually show a bit more: there is a single $\Sigma_1$-formula $\f$ such that $\f(\0), \f(\1), \dots $ is the series of formulas witnessing that $R^\backsim$ is Rosser.\footnote{The version of the following lemma concerning extensions of $R$ over classical logic was established by Putnam and Smullyan (1960).}

\begin{lem}\label{l:rosser}
For each pair of disjoint r.e.\ sets $A,B \subseteq \mathbb{N}$, there is a $\Sigma_1$-formula $\f(x)$ such that

\begin{center}
$n\in A$ implies $R^\backsim \vdash \f(\on)$\;\,

$n\in B$ implies $R^\backsim \vdash \neg\f(\on)$
\end{center}

\end{lem}

\begin{proof}
From the recursion theory (see e.g. Lindstr\"{o}m (1997), Fact~1.3(b)) we know that $A$ and $B$ are definable using the \emph{classical} $\Sigma_1$-formulas, i.e., formulas in the language with functions $+$ and $\cdot$ instead of predicates $A$ and $M$. Let us show that each `classical' $\Delta_0$-formula is equivalent to a $\Delta_0$-formula in our language (assuming that $\mathbb{N}$ interprets all these symbols).

Let us call the terms and formulas of our language \emph{simple}. An classical atomic formula is \emph{almost-simple} if it is simple or of the form $x = t_1 \circ t_2$, where $\circ$ is either $+$ or $\cdot$, $x$ is a variable and $t_1$ and $t_2$ are simple term. Clearly replacing such a formula by $A(t_1, t_2, z)$ or $M(t_1, t_2, x)$ respectively yields an equivalent formula in our language. So it suffices to show that any `classical' $\Delta_0$-formula is equivalent to a `classical' $\Delta_0$-formula where all atomic formulas are simple or almost-simple (we omit the adjective `classical' from now on). First, we observe the validity of the following three statements of classical logic for $\circ$ being each $+$ or $\cdot$, terms $t, t_1, t_2$ and variables $x, x_1, x_2$ not occurring in those terms:
\begin{itemize}
\item $\mathbb{N}\models  t = t_1 \circ t_2   \leftrightarrow  (\exists x_1\leq t)(\exists x_2\leq t) (t = x_1\circ x_2)$
\item $\mathbb{N}\models  t_1 \circ t_2  \leq t   \leftrightarrow  (\exists x_1\leq t)(\exists x_2\leq t)(x_1\circ x_2 \leq t)$
\item $\mathbb{N}\models  t \leq t_1 \circ t_2   \leftrightarrow  (\exists x_1\leq t)(\exists x_2\leq t) (x_1\leq t_1 \wedge x_2\leq t_2 \wedge t = x_1\circ x_2 )$.
\end{itemize}
Next we note that applying any of these equivalencies to any non-almost-simple atomic subformula of a given $\Delta_0$-formula $\x$ strictly decreases the finite multiset of depths of terms occurring in the formula according to the standard multiset well-ordering.\footnote{A finite multiset over a set $S$ is an ordered pair $\langle S,f \rangle$, where $f$ is a function $f\colon S \to \mathbb{N}$ and $\{ x \in S  \mid  f(x) > 0 \}$ is finite.  If $\le$ is a well-ordering of $S$, then:
\[
\langle S,f \rangle \leq_m \langle S,g \rangle \quad :\Longleftrightarrow\quad \forall x \in S \bigl( f(x) > g(x) \ \Longrightarrow \ \exists y \in S \bigl( y >  x \text{ and } g(y) > f(y)\bigr) \bigr)
\]
is a well-ordering on the set of all finite multisets over $S$, known as the Dershowitz--Manna ordering, for which see Dershowitz and Manna (1979).} Therefore exhaustively applying these equivalencies yields the required equivalent $\Delta_0$-formula with only almost-simple atomic formulas.

Thus we can assume that there are $\Delta_0$-formulas $\alpha(x,v)$ and $\beta(x,v)$ (in our language) such that
\begin{itemize}
\item $n\in A$ iff $\mathbb{N} \models \Exi{v} \alpha(n,v)$.
\item $n\in B$ iff $\mathbb{N} \models \Exi{v} \beta(n,v)$.
\end{itemize}
We define the $\Delta_0$-formula $\p(x)$:
$$
\p(x,v) = \neg(\neg\alpha(x,v) \vee \Exi{u\leq v} \beta(x,u))
$$
and show that the $\Sigma_1$-formula $\f(x) = \Exi{v} \p(x,v)$ has the desired properties.

The first case ($n\in A$) is easy: observe that this entails that $n\notin B$ and so we have not only $\mathbb{N} \models \alpha(n,m)$ for some $m$ but also that $\mathbb{N} \models \neg \beta(n, k)$ for each $k$ and so $\mathbb{N} \models  \neg\Exi{u\leq m} \beta(n,u)$. Therefore $\mathbb{N}\models \f(n)$ and so by Theorem~\ref{t:Sigma-comp} it follows that $R^{\backsim} \vdash \f(\on)$.

The proof of the second case ($n\in B$) is not so direct because $\neg\f$ is not a $\Sigma_1$-formula and so we cannot use $\Sigma_1$-completeness directly. However, because we know that $\mathbb{N} \models  \beta(n, m)$  for some $m$ and that $\mathbb{N} \models \neg \alpha(n, k)$ for each $k$ we can use it to obtain $R^{\backsim} \vdash \beta(\on,\om)$ and $R^{\backsim} \vdash \neg\alpha(\on,\ok)$. Using these facts we prove that
\begin{eqnarray}
R^{\backsim} \vdash v \leq \om \to \neg\alpha(\on,v)  \\[1ex]
R^{\backsim} \vdash \om \leq v \to \Exi{u\leq v} \beta(\on,u).
\end{eqnarray}
and after establishing these two claims the proof that $R^\backsim \vdash \neg\f(\on)$ easily follows: indeed using them together with ($\vee$intro), ($\vee$elim), and $(R^\backsim5)$ we obtain
$$
R^{\backsim} \vdash  \neg\alpha(\on,v) \vee \Exi{u\leq v} \beta(\on,u)
$$
and so using (dni) we get $R^{\backsim} \vdash \p(x,v) \to \bot$ and so ($\exists$elim) completes the proof.

To prove (1) we start with $R^{\backsim} \vdash \neg\alpha(\on,\ok)$ and (aux) to obtain for any $k\leq m$:
$$
R^{\backsim} \vdash v = \ok \to \neg \alpha(\on,v)
$$
Therefore the proof is done thanks to ($\vee$elim) and $(R^\backsim4)$.

To prove (2) we use the claim $R^{\backsim} \vdash  \beta(\on,\om)$ together with (weakening), (identity), and ($\vee$intro) to obtain:
$$
R^{\backsim} \vdash \om \leq v \to \om\leq v \wedge \beta(\on,\om)
$$
and so ($\exists$intro) completes the proof.
\end{proof}

All that is left is to apply Theorem~\ref{t:RosGen} and Proposition~\ref{p:UndecUpward} to obtain the following results.

\begin{thm}
The theory $R^{\backsim}$ is essentially undecidable in $\PLog$.
\end{thm}

\begin{cor} \label{maincor}
Let $\logic{L}$ be a propositional logic expanding $\Log$, $\LOG$ a predicate logic expanding $\plogic{L}$, and $T$ a theory strengthening  $R^{\backsim}$ in \LOG. If $T$ is $\LOG$-consistent (i.e.\ proves $\f$ and $\neg\f$ for no $\Sigma_1$-formula $\f$), then it is essentially undecidable in $\LOG$.
\end{cor}

\section{A stronger logic}\label{s:stronger-logic}

Our stronger propositional logic in the propositional language $\lang{L}_0$, which we denote as $\DSL$---short for ``the $\lang{L}_0$-fragment of the non-associative Lambek calculus with left and right weakening''---is given by the following Hilbert style proof system (note here that $\DSL$, unlike $\Log$, is not paraconsistent as we include the axiom ($\bot$elim)):

\medskip

\begin{tabular}{l@{\quad}l}
 (identity) & $\f\to\f$
\\ ($\wedge$elim) & $(\f\wedge\p)\to\f$
\\ ($\wedge$elim) & $(\f\wedge\p)\to\p$ \hfill
\\ ($\wedge$intro) & $((\f\to\p)\wedge(\f\to\x))\to(\f\to\p\wedge\x)$
\\ ($\vee$intro)& $\f\to(\f\vee\p)$
\\ ($\vee$intro)& $\p\to(\f\vee\p)$ \hfill
\\ ($\vee$elim)& $(\f\to\x)\wedge(\p\to\x)\to(\f\vee\p\to\x)$
\\ (weakening) & $\f\to(\p\to \f)$ \hfill
\\ ($\bot$elim) & $\bot\to\f$ \hfill
\\ (MP) & $\f,\f\to\p\wdash\p$
\\ (adj) & $\f,\p\wdash\f\wedge\p$
\\ (tone$\to$) & $\f\to\p,\x\to\theta\wdash(\p\to\x)\to(\f\to\theta)$ \hfill
\\ (assertion) & $\f\wdash(\f\to\p)\to\p$
\end{tabular}

Let us first observe that $\DSL$ indeed extends $\Log$: the latter's rules of (weakening), ($\wedge$intro), and ($\vee$elim) follow from the corresponding axioms of $\DSL$ using the rules (MP) and (adj); (trans) follows from (tone$\to$) by taking $\p=\x$ using (identity) and (MP); and the rule (morg) follows from the axiomatic form stated below.

Let us note that the rule (tone$\to$) can be equivalently replaced by the following two rules:\footnote{For one direction juts consider suitable instances of (tone$\to$) and the axiom (identity); for converse direction first use (prefixing) to obtain $\x\to\theta\wdash(\p\to\x)\to(\p\to\theta)$ and then (suffixing) to obtain $\f\to\p\wdash(\p\to\theta)\to(\f\to\theta)$ and the rule (trans) completes the proof.}

\smallskip

\begin{tabular}{l@{\quad}l}
(suffixing) & $\f\to\p\wdash(\p\to\x)\to(\f\to\x)$  \\
(prefixing) & $\f\to\p\wdash (\x\to\f)\to(\x\to\p)$
\end{tabular}

\medskip

\noindent The following theorems/rules are derivable in $\Log$:\footnote{The first rule is a direct consequence of (suffixing); both (morg)s are consequences of ($\land$intro), ($\land$elim), ($\lor$intro), ($\lor$elim), and (cont); the stronger form of (weakening) follows from applying  ($\land$intro) twice on $\f\to(\p\to \p)$ and $\f\to(\p\to \f)$; (red) follows from applying (prefixing) on ($\bot$elim); and finally to obtain (exp) apply (prefixing) twice on $\f\wedge\p\to\x$ and use (weakening).}

\smallskip
\begin{tabular}{l@{\quad}l}
(cont) & $\f\to\p\wdash\neg\p\to\neg\f$.\\
(morg) & $\neg (\f \vee \p) \leftrightarrow \neg \f \wedge \neg \p$ \\
(morg) & $\neg \f \vee \neg \p \to \neg (\f \wedge \p)$ \\
(weakening) & $\f\to(\p\to \f\land\p)$ \\
(red) & $\neg\f \to (\f\to \x)$ \\
(exp) & $\f\wedge\p\to\x\wdash \f\to(\p\to\x)$\\
\end{tabular}

\begin{remark}
\DSL\ can be seen as a fragment of the non-associative Lambek calculus with left and right weakening (in the terminology of Galatos et al (2007) ($\bot$elim) is \emph{left} weakening and our (weakening) is their\/ \emph{right} weakening). It is indeed just a fragment: the language of the \emph{full} logic \DSL\ also involves fusion (residuated conjunction) and dual implication and it is well known that they are not definable from our connectives. The non-associativity of \DSL\ refers to the residuated conjunction, but this fact can be expressed, using implication and their statement of residuation, as the failure of formula $(\f\to\p)\to((\p\to\x)\to(\f\to\x))$. Therefore $\DSL$ is strictly weaker than H\'{a}jek's logic~{\rm BL}.

$\DSL$ can be seen as the extension of the positive fragment (without distribution) of the basic relevant logic \textbf{B}, studied, for instance, by Routley et al (1982), by the weakening axiom, the assertion rule, and negation defined in terms of $\bot$.
\end{remark}

Finally we need to prove two important facts about \emph{crisp} formulae, i.e.\ formulae $\p$ where $\p\lor\neg\p$ is provable. The first claim can be seen as converse of the derived rule (exp).


\begin{prop}\label{p:Crispness}
Assume that formula $\f$ is crisp in $T$ and $T\vdash \f\to (\p\to\x)$. Then also $T\vdash \f\wedge\p\to\x$. If furthermore $\p$ is also a crisp formula in $T$, then the formula $\f\vee\p$ is crisp in $T$ as well.
\end{prop}

\begin{proof}
We present formal derivations of both claims
\begin{itemize}
\item \begin{prooflist}
\item[(1)] $(\p\to\x)\to (\f\wedge\p\to\x)$ \hfill ($\wedge$elim) and (suffixing)
\item[(2)] $(\f\to\x)\to (\f\wedge\p\to\x)$ \hfill ($\wedge$elim) and (suffixing)
\item[(3)] $\f\to(\f\wedge\p\to\x)$ \hfill $\f\to (\p\to\x)$, (1), and (trans)
\item[(4)] $\neg\f\to(\f\wedge\p\to\x)$ \hfill (red), (2), and (trans)
\item[(5)] $\f\land\p\to\x$ \hfill (3), (4), ($\vee$elim) and crispness of $\f$.
\end{prooflist}

\pagebreak

\item Let us denote the formula $(\f\lor\p)\lor\neg(\f\lor\p)$ as $\x$
\begin{prooflist}
\item[(1)] $\f\to\x$ \hfill ($\lor$intro)
\item[(2)] $\p\to\x$ \hfill ($\lor$intro)
\item[(3)] $\neg\f\land\neg\p\to \x$ \hfill (morg), ($\lor$intro), and (trans)
\item[(4)] $\neg\f\to(\neg\p\to\x)$ \hfill(3) and (exp)
\item[(5)] $\f\to(\neg\p\to\x)$ \hfill (1), (weakening), and (trans)
\item[(6)] $\neg\p\to\x$ \hfill (4), (5), ($\lor$elim), and crispness of $\f$
\item[(7)] $\x$ \hfill (2), (6), ($\lor$elim), and  crispness of $\p$
\end{prooflist}
\end{itemize}
\vspace{-2.5ex}
\mbox{ }
\end{proof}

\section{$Q^\backsim$ strengthens $R^\backsim$ in $\PDSL$}\label{s:QandR}

In this section we prove that the arithmetical theory $Q^\backsim$ proves all theorems of $R^\backsim$ against the background of $\PDSL$. As a reminder, $Q^\backsim$ stands to $Q$ as $R^\backsim$ stands to $R$ and is axiomatised as follows:

\medskip

\begin{tabular}{l@{\ \ }l}
$(Q^\backsim0)$ & $x=y\lor x\neq y$
\\$(Q^\backsim1)$ &  $S(x) \neq \0$
\\ $(Q^\backsim2)$ & $S(x) =S(y) \to x=y$
\\ $(Q^\backsim3)$ & $x \neq \0  \to  \Exi{y} (x=S(y))$
\\ $(Q^\backsim4)$ & $A(x, \0, y)  \leftrightarrow  x=y$
\\ $(Q^\backsim5)$ & $A(x, S(y), z)  \leftrightarrow  \Exi{u}(A(x, y, u) \wedge z=S(u) )$
\\ $(Q^\backsim6)$ & $M(x, \0, y)  \leftrightarrow  y=\0$
\\ $(Q^\backsim7a)$ & $ M(x, S(y), z)\to \Exi{u}(M(x, y, u) \wedge A(u, x, z))$
\\ $(Q^\backsim7b)$ & $M(\om,\on,u)\to(A(u,\on, x)\to M(\om,\overline{n+1},x))$
\\ $(Q^\backsim8)$ & $x\leq y \leftrightarrow  \Exi{z}A(z, x, y)$
\end{tabular}

\medskip

\begin{remark}
It is noteworthy that our $Q^\backsim$ differs slightly from H\'{a}jek's. There are two ways in which this is the case. First, we include the additional axiom ($Q^\backsim$0) stating that identities are crisp. H\'{a}jek includes this as an assumption of the first order logic, whereas we build it directly into the theory.

In addition, his system includes only one axiom ($Q^\backsim7$)---the biconditional version of our $(Q^\backsim7a)$---and furthermore in his $(Q^\backsim5)$ and ($Q^\backsim$7), the conjunction occurring is the strong conjunction of {\rm BL} (that of which the conditional is residual). Note that all of our $Q^\backsim$ axioms are provable from H\'{a}jek's version of the theory in {\rm BL}. First, $(Q^\backsim7a)$ and $(Q^\backsim7b)$ are consequences of his version stated with strong conjunction. In addition, one can prove our $(Q^\backsim5)$ from his arithmetic theory in {\rm BL}. First, since $\f\to(\p\to\f\land\p)$ is provable in both our systems, our left-to-right direction of $(Q^\backsim5)$ is an immediate consequence of his (this is the same reason as that for why our $(Q^\backsim7a)$ is a consequence of his axiom). Second, $(\exists u)(A(x,y,u)\land z=S(u))\to A(x,S(y),z)$ is provable in his system given the crispness of identity, as he shows that in first order {\rm BL} (which is actually strictly stronger then $\plogic{{BL}}$ by using the additional axiom of constants domains), whenever $\f$ is crisp, then any weak conjunction of $\f$ with some other formula is equivalent to their strong conjunction; see H\'ajek (2007), remark 2.1(1). So, since $z=S(u)$ is crisp, the result follows.

Hence, our results do indeed generalise H\'ajek's, despite our using a variant on his $Q^\backsim$.
\end{remark}

Note that in $Q^\backsim$, thanks to ($Q^\backsim$0), Prop~\ref{p:Crispness}, and (exp), we can replace the identity axiom ($=$prin) for formulas by its equivalent formulation:

\smallskip

\begin{tabular}{l@{\ \ }l}
($=$prin) &  $x=y\wedge \f(x) \to\f(y)$
\end{tabular}

\begin{thm}\label{fund}
$Q^\backsim$ strengthens $R^\backsim$ in $\emph{Q}\DSL$.
\end{thm}

\begin{proof}
Before we start proving the axioms of $R^\backsim$ let us prove two useful preliminaries:
\begin{description}
\item[Claim 1] $Q^\backsim \vdash x\leq y \leftrightarrow  S(x) \leq S(y))$.

\item[Claim 2] For any formula $\f(x)$ such that $Q^\backsim \vdash \f(\0)$ and $Q^\backsim \vdash \f(S(y))$ we have $Q^\backsim \vdash \f(x)$.
\end{description}

To prove the first claim it clearly suffices to prove $A(x,y,z) \leftrightarrow  A(x,S(y),S(z))$ and use ($\exists$intro), ($\exists$elim) and $(Q^\backsim8)$ to complete the proof. The proof of the left-to-right implication is easy: clearly $A(x,y,z) \to A(x,y,z)\wedge S(z) = S(z)$. Thus, by ($\exists$intro) and (trans), $A(x,y,z) \to \Exi{u}(A(x,y,u)\wedge S(z) = S(u))$ and so $(Q^\backsim5)$ completes the proof. The converse implication is a bit more complex:
\begin{prooflist}
\item[(1)] $S(z)= S(u) \to z = u$  \hfill   ($Q^\backsim2$)
\item[(2)] $A(x, y, u)\wedge S(z)=S(u) \to A(x, y, u)\wedge z = u $ \hfill  (1), ($\land$intro), ($\land$elim)
\item[(3)] $A(x, y, u)\wedge S(z)=S(u) \to A(x, y, z)$  \hfill   (2), ($=$prin), and (trans)
\item[(4)] $\Exi{u}(A(x, y, u)\wedge S(z)=S(u)) \to A(x, y, z)$  \hfill   (3) and ($\exists$elim)
\item[(5)] $A(x,S(y),S(z)) \to A(x, y, z) $  \hfill   ($Q^\backsim5$), (4), (trans)
\end{prooflist}

To prove the second claim let us use (aux) for both premises to obtain $Q^\backsim \vdash x = 0 \to \f(\0)$ and $Q^\backsim \vdash x = S(y) \to \f(x)$. Using ($\exists$elim)  and $(Q^\backsim3)$ we obtain $Q^\backsim \vdash x \neq \0 \to \f(x)$ and thus ($\vee$elim) and $(Q^\backsim0)$ complete the proof.

\medskip

$(R^\backsim1)$: First we prove $A(\om, \on, x) \to x= \overline{m+n}$ by metainduction on $n$. The case $\on = \0$ follows from ($Q^\backsim4$).  The inductive case:

\begin{prooflist}
\item[(1)] $A(\om, \on, u) \to u = \overline{m+n}$ \hfill by IH

\item[(2)] $A(\om, \on, u) \wedge x = S(u) \to u = \overline{m+n} \wedge x = S(u)$ \hfill (1), ($\wedge$intro), ($\wedge$elim)


\item[(3)] $A(\om, \on, u) \wedge x = S(u) \to x =  \overline{m+n+1}$ \hfill (2), ($=$prin), (trans)

\item[(4)] $\Exi{u}(A(\om, \on, u) \wedge x = S(u)) \to x =  \overline{m+n+1}$ \hfill (3), ($\exists$elim)

\item[(5)] $A(\om, \onn, u)  \to x =  \overline{m+n+1}$ \hfill (4), ($Q^\backsim5$), (trans)
\end{prooflist}

To prove the converse direction we show, again by metainduction on $n$, that $A(\om, \on, \overline{m+n})$ and use (aux) to complete the proof. Again, the case $\on = \0$ follows from ($Q^\backsim4$) and the inductive case follows immediately from the induction assumption and the fact that $A(x,y,z) \leftrightarrow  A(x,S(y),S(z))$ which we established in the proof of Claim 1.

\medskip

$(R^\backsim2)$: The proof is similar; first establish $M(\om, \on, x) \to x= \overline{m\cdot n}$ by metainduction on $\on$. The case $\on = \0$ follows from $(Q^\backsim6)$. The inductive case:

\begin{prooflist}
\item[(1)] $M(\om, \on, u) \to u = \overline{m\cdot n}$ \hfill by IH

\item[(2)] $M(\om, \on, u) \wedge A(u,\om,x) \to u = \overline{m\cdot n} \wedge A(u,\om,x)$ \hfill (1), ($\land$intro), ($\land$elim)

\item[(3)] $M(\om, \on, u) \wedge A(u,\om,x) \to  A(\overline{m\cdot n},\om,x)$ \hfill (2), ($=$prin), (trans)




\item[(4)] $M(\om, \on, u) \wedge A(u,\om,x) \to x =  \overline{m\cdot (n+1)}$ \hfill (3), $(R^\backsim1)$, (trans)

\item[(5)] $\Exi{u}(M(\om, \on, u) \wedge A(u,\om,x)) \to x =  \overline{m\cdot (n+1)}$ \hfill (4), ($\exists$elim)

\item[(6)] $M(\om, \onn, u)  \to x =  \overline{m\cdot (n+1)}$ \hfill (5), ($Q^\backsim7a$), (trans)
\end{prooflist}

To prove the converse direction we show, again by metainduction on $n$, that $M(\om, \on, \overline{m\cdot n})$ and then (aux) completes the proof. Again, the case $\on = \0$ follows from $(Q^\backsim6)$. The inductive case:

\begin{prooflist}
\item[(1)] $ M(\om, \on, \overline{m\cdot n})$ \hfill IH

\item[(2)] $ A(\overline{m\cdot n},\on,\overline{m\cdot (n+1)})$ \hfill $(R^\backsim1)$

\item[(3)] $M(\om, \on, \overline{m\cdot n}) \to (A(\overline{m\cdot n},\on,\overline{m\cdot (n+1)}) \to M(\om, \onn, \overline{m\cdot (n+1)}))$ \hfill ($Q^\backsim7b$)

\item[(4)] $M(\om, \onn, \overline{m\cdot (n+1)})$ \hfill  (3) and  MP twice
\end{prooflist}

\medskip

$(R^\backsim3)$: It suffices to establish the case where $n < m$. Observe that $\om = \on \to \overline{m-n} = \0$ by repeated use of $(Q^\backsim2)$. As $m - n \neq 0$ we know that $\om = \on \to S(\overline{m-n-1}) = \0$ and as $S(\overline{m-n-1}) = \0 \to \bot$ due to ($Q^\backsim1$), the claim follows.

\medskip

$(R^\backsim4)$: To prove the right-to-left direction observe that for $k \leq n$ we have $A(\overline{n - k},\ok,\on)$ due to $(R^\backsim1)$
and so $\ok\leq \on$ using ($\exists$intro) and $(Q^\backsim8)$ and thus $x = \ok \to x\leq \on$ by (aux). Repeated use of $\vee$elim) then completes the proof of this direction.

To prove the converse direction set $\f_n = x\leq \on \to x=\0 \vee x=\1 \vee \dots \vee x=\on$ and we prove $\f_n$ by metainduction over $n$.

For the base case, $x\leq\0\to x=\0$, we employ Claim 2. First note that $\0\leq\0\to\0=\0$ follows from (weakening). Next, $(Q^\backsim1)$ gives us that $\0\neq S(u)$ and so by ($\wedge$elim) and (cont), $\neg(\exists u)(A(z,y,u)\land\0=S(u))$ holds. By $(Q^\backsim5)$, it follows that $\neg A(z,S(y),\0)$, and thus ($\bot$elim), (trans), ($\exists$elim), and $(Q^\backsim8)$ entail that $S(y)\leq\0\to S(y)=\0$. So Claim 2 delivers the desired result.

Next, observe that thanks to ($\vee$intro) and (weakening) we have $\f_{n+1}(\0)$ and so if we prove  $\f_{n+1}(S(y))$ the claim follows using Claim 2.

\begin{prooflist}
\item[(1)] $S(y) \leq \onn \to y \leq \on$ \hfill Claim 1
\item[(2)] $S(y) \leq \onn \to y = \0 \vee y =\1\vee \dots y = \on$ \hfill IH, (2), and (trans)
\item[(3)] $y = \0 \vee y =\1\vee \dots y = \on \to S(y) = \0 \vee S(y) = \1 \vee \dots S(y) = \onn$
\item[]    \mbox{}     \hfill repeated use of ($=$prin), ($\vee$elim), and ($\vee$intro).
\item[(4)] $\f_{n+1}(S(y))$ \hfill (2), (3), and (trans)
\end{prooflist}

\medskip

$(R^\backsim5)$: Let us set $\f_n = x\leq \on \vee \on\leq x$ and we prove $\f_n$ by metainduction over $n$. Clearly from ($Q^\backsim4$) and $(Q^\backsim8)$ we get that $\0 \leq x$ and so by ($\vee$intro) we obtain both the base case and also  $\f_{n+1}(\0)$. Thus again proving $\f_{n+1}(S(y))$ completes the prof due to Claim 2.
\begin{prooflist}
\item[(1)] $y \leq \on \to S(y) \leq \onn $ \hfill Claim 1
\item[(2)] $\on \leq y \to \onn \leq S(y) $ \hfill Claim 1
\item[(3)] $y\leq \on \vee \on\leq y \to S(y)\leq \onn \vee \onn\leq S(y)$ \hfill (2), (3),  ($\vee$elim), and ($\vee$intro)
\item[(4)] $\f_{n+1}(S(y))$ \hfill (3), IH, and (trans)
\end{prooflist}

\medskip

$(R^\backsim6)$: Thanks to $(R^\backsim4)$ we know that $x\leq \on$ is equivalent to a disjunction of crisp formulas and so it is crisp as a result of Proposition~\ref{p:Crispness}.
\end{proof}

As before, it is obvious that the structure $\mathbb{N}$ can be interpreted as a model of $Q^{\backsim}$ in Q$\DSL$, hence $Q^{\backsim}$ is consistent in $\PDSL$. Therefore the previous theorem and Corollary~\ref{maincor}, allows us to prove the following theorem, which can indeed be seen as a generalisation of H\'{a}jek's result in first order {\rm BL}.

\begin{cor}
$Q^\backsim$ is essentially undecidable in Q$\DSL$.
\end{cor}

\section{Concluding remarks}

We have shown that the weak arithmetic theory $R^\backsim$ is essentially undecidable against the background of the weak propositional logic $\Log$ extended by minimal first-order axioms. The first upshot of this is that the cost of entry for essential undecidability is very low indeed -- one needs only a fairly weak arithmetic theory and a fairly weak logic. Furthermore, we can show that $R^\backsim$ is a weaker theory than even the very weak $Q^\backsim$ in the context of a slightly stronger (but still quite weak) logic. This extends and strengthens H\'ajek's result and suggests avenues of further investigation, perhaps using Smullyan's representation systems, into the limits of undecidability in mathematical theories against the background provided by weak logics. 

\section{Acknowledgments}

P.\ Cintula was supported by the project GA17-04630S of the Czech Science Foundation (GA\v{C}R) and by RVO 67985807. A.\ Tedder was supported by the GA\v{C}R project 18-19162Y. Thanks are due to an anonymous referee for helpful comments. This paper was presented at the Melbourne Logic Seminar, the conference Logic Colloquium in Prague, and the conference Services to Logic: 50 Years of the Logicians' Liberation League in Mexico City. We are grateful to the audiences in all these venues. Finally, Albert Visser provided some useful comments on an earlier version of this work.


\

\

\address{SCHOOL OF HISTORICAL AND PHILOSOPHICAL INQUIRY\\
\hspace*{9pt}UNIVERSITY OF QUEENSLAND\\
\hspace*{18pt}ST LUCIA QLD 4072, AUSTRALIA\\
{\it E-mail}: guillebadia89@gmail.com\\[8pt]\\
\hspace*{9pt}THE INSTITUTE OF COMPUTER SCIENCE OF THE CZECH ACADEMY OF SCIENCES\\
\hspace*{18pt}PRAGUE 8, 182 00, CZECH REPUBLIC\\
{\it E-mail}: cintula@cs.cas.cz\\
{\it E-mail}: ajtedder.at@gmail.com}
\clearpage

\end{document}